\documentclass[a4paper,12pt]{article}

\usepackage[utf8]{inputenc}
\usepackage[english]{babel}
\usepackage{enumitem}
\usepackage{float}
\usepackage{amsmath}
\usepackage{amssymb}
\usepackage{amsthm}
\usepackage{amsfonts}
\usepackage{booktabs}
\usepackage{epsfig}
\usepackage{epstopdf}
\usepackage{cite}
\usepackage{multirow} % multirow for format of table
\usepackage{algorithm,algpseudocode}
\usepackage{subcaption}
\usepackage{xcolor}
\usepackage[linkcolor=blue,colorlinks=true]{hyperref}

\usepackage{caption}
\newtheorem{dfn}{Definition}[section]
\newtheorem{lem}[dfn]{Lemma}
\newtheorem{thm}[dfn]{Theorem}
\newtheorem{cor}[dfn]{Corollary}

\theoremstyle{definition}

\newtheorem{asm}[dfn]{Assumption}
\newtheorem{rem}[dfn]{Remark}

\allowdisplaybreaks

\title{Linear Convergence Results for Inertial Type Projection Algorithm for Quasi-Variational Inequalities}

\date{\today}

\author{Yonghong Yao \footnote{School of Mathematical Sciences, Tiangong University, Tianjin 300387, China; and Center for Advanced Information Technology, Kyung Hee University, Seoul 02447, South Korea; e-mail: yyhtgu@hotmail.com}
\hspace*{0.8mm}, Lateef O. Jolaoso\footnote{School of Mathematical Sciences, University of Southampton, SO17 1BJ, United Kingdom; Department of Mathematics and Applied Mathematics, Sefako Makgatho Health Sciences University, P.O. Box 94 Medunsa 0204, Pretoria, South Africa; e-mail: l.o.jolaoso@soton.ac.uk.}
\hspace*{0.8mm}, Yekini Shehu\footnote{(Corresponding Author) School of Mathematical Sciences, Zhejiang Normal University, Jinhua 321004, People’s Republic of China; e-mail: yekini.shehu@zjnu.edu.cn}}
\begin{document}

\maketitle

\begin{abstract}
\noindent Many recently proposed gradient projection algorithms with inertial extrapolation step for solving quasi-variational inequalities in Hilbert spaces are proven to be strongly convergent with no linear rate given when the cost operator is strongly monotone and Lipschitz continuous. In this paper, our aim is to design an inertial type gradient projection algorithm for quasi-variational inequalities and obtain its linear rate of convergence. Therefore, our results fill in the gap for linear convergence results for inertial type gradient projection algorithms for quasi variational inequalities in Hilbert spaces. We perform numerical implementations of our proposed algorithm and give numerical comparisons with other related inertial type gradient projection algorithms for quasi variational inequalities in the literature.\\



\end{abstract}

\section{Introduction}
Suppose that $H$ is a real Hilbert space equipped with inner product $\langle \cdot ,\cdot \rangle $ and induced norm $\Vert \cdot \Vert $. Assume further that $K$ is a nonempty, closed and convex subset of $H$. Given a nonlinear operator $\mathcal{A}:H\rightarrow H$  and a set-valued operator $K:H\rightrightarrows H$ such that for each element $x\in H$, we have a closed and convex set $K(x)\subset H$. Then the  \textit{Quasi-Variational Inequality} (QVI), is to find $x_{*}\in H$ such that $x_{*}\in K(x_{*})$
and%
\begin{equation}
\left\langle \mathcal{A}(x_{*}),x-x_{*}\right\rangle \geq0\text{ for all
}x\in K(x_{*}). \label{eq:QVI}%
\end{equation}
\noindent
In the special case when $K(x)\equiv K$ for all $x\in H$, we have the QVI \eqref{eq:QVI} reduces to the variational inequality problem (\cite{Fichera63,Fichera64,Stampacchia64,KS80}), viz.: find $x_{*}\in K$ such that%
\begin{equation}
\left\langle \mathcal{A}(x_{*}),x-x_{*}\right\rangle \geq0\text{ for all
}x\in K. \label{eq:VIP}%
\end{equation}

\noindent
Several projection type methods have been introduced to solve QVI \eqref{eq:QVI} in the literature. In \cite{ajm18},
 Antipin et al. designed the gradient projection algorithm:
\begin{equation}\label{extra}
 x_{k+1}= P_{K(x_k)}(x_k-\gamma \mathcal{A}(x_k))
\end{equation}
and the extragradient algorithm:
		\begin{eqnarray} \label{extra2}
\left\{  \begin{array}{ll}
      &	y_k= P_{K(x_k)}(x_k-\gamma \mathcal{A}(x_k)),\\
      & x_{k+1}= P_{K(x_k)}(x_k-\gamma \mathcal{A}(y_k))
      \end{array}
       \right.
      \end{eqnarray}
\noindent
 to solve
QVI \eqref{eq:QVI} and obtained strong convergence results when $K(x):=c(x)+K$, $c$ Lipschitz continuous, $\mathcal{A}$ is strongly monotone and Lipschitz continuous. Similar results to \cite{ajm18} are also obtained in \cite{Mosco76,noor85,noor88,nnk13}.\\

\noindent
We note that the extragradient method \eqref{extra2} requires computing two projections and two evaluations of $\mathcal{A}$ per iteration. This could be computational expensive for large scale problems. In \cite{Mijajlovic2018}, Mijajlovi\'{c} et al. introduced  relaxed projection type algorithms
\begin{equation}\label{extra3}
x_{k+1}=(1-\alpha_k)x_k+\alpha_k P_{K(x_k)}(x_k-\gamma \mathcal{A}(x_k))
\end{equation}
 and
		\begin{eqnarray}\label{extra4}
\left\{  \begin{array}{ll}
      &	 y_k= (1-\beta_k)x_k+\beta_k P_{K(x_k)}(x_k-\gamma \mathcal{A}(x_k)),\\
      &  x_{k+1}= (1-\alpha_k)x_k+\alpha_k P_{K(y_k)}(y_k-\gamma \mathcal{A}(y_k))
      \end{array}
       \right.
      \end{eqnarray}
where $\alpha_k \in (0,1], \beta_k \in [0,1]$. Consequently, strong convergence results are obtained for QVI \eqref{eq:QVI} when  $\mathcal{A}$ is strongly monotone and Lipschitz continuous operator with condition \eqref{eq: lambda condition} fulfilled for which the case $K(x):=c(x)+K, x \in H$ satisfies. The proposed algorithm in \eqref{extra} is a special case of algorithm \eqref{extra3} with $\alpha_k=1$ and the results of Mijajlovi\'{c} et al. \cite{Mijajlovic2018} are also related to the ones in  \cite{Aussel,Facchinei1,Facchinei2,Latorre}.\\

\noindent
Due to recent contributions of optimization algorithms with inertial extrapolation steps in terms of improvement on the speed of convergence as enumerated in \cite{Attouch,Attouch2,alvarez,Alvarez,Beck,Mainge2,Polyak2,Attouch3,Lorenz,Ochs,Bot,Shehu,Bot2,CChen,Bot3} and other related papers, gradient projection algorithms with extrapolation step for solving QVI \eqref{eq:QVI} was studied in \cite{ShehuOptim}:
\begin{eqnarray}\label{extra5}
\left\{\begin{array}[c]{ll}
 & y_{k}=x_{k}+\theta_k(x_{k}-x_{k-1}),\\
 & x_{k+1}=(1-\alpha_k)y_k+\alpha_k P_{K(y_k)}(y_k-\gamma \mathcal{A}(y_k)),
\end{array}\right.
\end{eqnarray}
with $\alpha_k \in (0,1)$ and $0\leq \theta_k \leq \theta<1$ and strong convergence results (with no linear convergence results) obtained under condition that $\mathcal{A}$ is strongly monotone and Lipschitz continuous. With $\mathcal{A}$ being strongly monotone and Lipschitz continuous in QVI \eqref{eq:QVI}, \d{C}opur et al. in \cite{Copur} also obtained strong convergence results for the gradient projection algorithm \eqref{extra} with double inertial extrapolation steps. However, no linear convergence results was also given in \cite{Copur}. Similar strong convergence results with no linear rate of convergence are given in \cite{Jabeen}. \\

\noindent
{\bf Our Contributions.}
\begin{itemize}
  \item Motivated by the inertial type gradient projection algorithms proposed in \cite{Copur,Jabeen,ShehuOptim}, where no linear convergence results are given, our aim in this paper is to prove the linear convergence of an inertial type projection algorithm for QVI \eqref{eq:QVI}.
  \item The inertial factor in our proposed algorithm has the possibility of taking both negative and non-negative values unlike the inertial factors in \cite{Copur,Jabeen,ShehuOptim} and other associated papers where the inertial factors are in [0,1]. Kindly see Remark \ref{oluwa} below.
  \item  We also show that
proposed inertial type gradient projection algorithm proposed in this paper is efficient and outperforms the other related inertial type gradient projection algorithms in \cite{Copur,Shehu,ShehuOptim} using standard test problems in the literature.
\end{itemize}

\noindent
\textbf{Outline.}
We outline the paper as, viz: Section \ref{sec:Preliminaries} entails some basic facts, concepts, and lemmas, which are needed in the linear convergence analysis. In Section \ref{Sec:Method}, we introduce an inertial type gradient projection algorithm with their corresponding linear convergence results given. Section \ref{Sec:Numerics} discusses the numerical implementations of the proposed algorithm with comparisons with other related algorithms while in Section \ref{Sec:Final}, we give a final conclusion of our results.

\section{Preliminaries}\label{sec:Preliminaries}

\begin{dfn}
Given an operator $\mathcal{A}:H\rightarrow H$,
\begin{itemize}
  \item $\mathcal{A}$ is called $L$-\emph{Lipschitz continuous} ($L>0$), if
\begin{equation}
\Vert \mathcal{A}(x)-\mathcal{A}(y)\Vert \leq L \Vert x-y\Vert \text{ for all } x,y\in H.
\end{equation}

\item  $\mathcal{A}$ is called $\mu$-\emph{strongly monotone} ($\mu>0$), if%
\begin{equation}
\langle \mathcal{A}(x)-\mathcal{A}(y),x-y\rangle\geq\mu\Vert x-y\Vert^{2}\text{ for all
}x,y\in H.
\end{equation}

\end{itemize}
\end{dfn}
\noindent
For each $x\in H$, there exists a unique nearest point in $K$, denoted
by $P_{K}(x)$, such that
\begin{equation}
\Vert x-P_{K}(x)\Vert \leq \Vert x-y\Vert \text{ for all } y\in K.
\label{eq:2.2}
\end{equation}%
This operator $P_{K}:H\rightarrow K$ is called the \textit{metric projection} of $H$
onto $K$, characterized \cite[Section 3]{GR84} by
\begin{equation}
P_{K}(x)\in K
\end{equation}
and
\begin{equation}
\left\langle x-P_{K}\left( x\right) ,P_{K}\left( x\right) -y\right\rangle
\geq 0\text{ for all }x\in H,\text{ }y\in K. \label{eq:2.3}
\end{equation}%

\noindent
We state the following sufficient conditions for the existence of solutions of QVIs \eqref{eq:QVI} given in \cite{noor94}.

\begin{lem}
  Let $\mathcal{A}:H\rightarrow H$ be $L$-Lipschitz continuous and $\mu$-strongly monotone on $H$ and $K(\cdot)$ be a set-valued mapping with nonempty, closed and convex values such that there exists $\lambda \geq 0$ such that
\begin{eqnarray}\label{eq: lambda condition}
\|P_{K(x)}(z)-P_{K(y)}(z)\| \leq \lambda \|x-y\|,~~x,y,z \in  H, \quad \lambda +\sqrt{1-\frac{\mu^2}{L^2}}<1.
\end{eqnarray}
Then the QVI \eqref{eq:QVI} has a unique solution.
\end{lem}
\noindent The fixed point formulation of the QVI \eqref{eq:QVI} is given by

\begin{lem}
Let $K(\cdot)$ be a set-valued mapping with nonempty, closed and convex values in $H$. Then $x_{*} \in K(x_{*})$ is a solution of the QVI \eqref{eq:QVI} if and only if for any $\gamma >0$ it holds that
$$
x_{*}=P_{K(x_{*})}(x_{*}-\gamma \mathcal{A}(x_{*})).
$$
\end{lem}

\noindent The following lemma is needed in our convergence analysis.

\begin{lem}\label{lm2}
If $x,y\in H$, we have
\begin{itemize}
\item[(i)]
 $2\langle x,y\rangle=\|x\|^{2}+\|y\|^{2}-\|x-y\|^{2}=\|x+y\|^{2}-\|x\|^{2}-\|y\|^{2}.$
\item[(ii)] Assume that $x,y\in H$ and $\alpha \in \mathbb{R}$. Then
\begin{eqnarray*}
\|\alpha x+(1-\alpha) y\|^2&=& \alpha\|x\|^2+(1-\alpha)\|y\|^2-\alpha(1-\alpha)\|x-y\|^2.
\end{eqnarray*}
\end{itemize}
\end{lem}

\section{Main Results}\label{Sec:Method}
\noindent
We introduce our inertial type gradient projection algorithm for solving QVI \eqref{eq:QVI} and obtain linear convergence results.  To begin with, let us assume that a parameter $\gamma \geq 0$ obeys the following condition:

\begin{asm}\label{Ass:VI}
\begin{equation}\label{eq: gamma condition}
\left|\gamma-\frac{\mu}{L^2}\right|<\frac{\sqrt{\mu^2-L^2\lambda(2-\lambda)}}{L^2}.
\end{equation}

\end{asm}

\noindent We now introduce an inertial type projection algorithm for solving the QVI \eqref{eq:QVI}.

\begin{algorithm}[H]
\caption{Inertial Type Gradient Projection Algorithm}\label{Alg:AlgL}
\begin{algorithmic}[1]
   \State  Choose $\theta_k \in (0,1)$ such that $0< a\leq \theta_k \leq b<1$ and pick
    $z_{-1}=z_0 \in H.$ Set $ k := 0 $.
   \State Given the current iterates $ z_{k-1}$ and $ z_k,$ compute
		\begin{eqnarray} \label{aa}
\left\{  \begin{array}{ll}
      &	y_{k-1}=z_k+\Big(\frac{1-2\theta_{k-1}}{\theta_{k-1}}\Big)(z_k-z_{k-1}),\\
      &x_k=P_{K(y_{k-1})}(y_{k-1}-\gamma \mathcal{A}(y_{k-1})),\\
      & z_{k+1}=(1-\theta_k)z_k+\theta_k x_k
      \end{array}
       \right.
      \end{eqnarray}
\State Set $ k\leftarrow k+1 $, and {\bf return to 2}.
\end{algorithmic}
\end{algorithm}

\begin{rem}\label{oluwa}$\,$\\
(a) Our proposed inertial-type gradient projection Algorithm \ref{Alg:AlgL} features an inertial factor $\frac{1-2\theta_{k-1}}{\theta_{k-1}} \in (-1,+\infty)$. In particular, if $\theta_{k-1} \in (\frac{1}{2},1)$ in Algorithm \ref{Alg:AlgL}, the inertial factor $\frac{1-2\theta_{k-1}}{\theta_{k-1}} \in (-1,0)$, while if $\theta_{k-1} \in (0,\frac{1}{2})$ in Algorithm \ref{Alg:AlgL}, the inertial factor $\frac{1-2\theta_{k-1}}{\theta_{k-1}} \in (0,+\infty)$.
Therefore, the inertial term in our algorithm \ref{Alg:AlgL} is not restricted to $[0,1]$, which was considered in other inertial-type projection algorithms proposed and studied in \cite{Copur,Jabeen,Shehu,ShehuOptim}.\\\\

\noindent
    (b) If $\theta_k=\frac{1}{3}$ in Algorithm \ref{Alg:AlgL}, then Algorithm \ref{Alg:AlgL} becomes an averaging reflected gradient projection algorithm, which is a similar algorithm studied in \cite{Shehu}.\\

\end{rem}

\noindent We now give our linear convergence results for Algorithm \ref{Alg:AlgL}.

\begin{thm}\label{th1}
Consider the QVI \eqref{eq:QVI} with $\mathcal{A}$ being $\mu$-strongly monotone and $L$-Lipschitz continuous and assume there exists $\lambda \geq 0$ such that \eqref{eq: lambda condition} holds.
Let $\{x_k\}$ and $\{z_k\}$ be generated by Algorithm \ref{Alg:AlgL} with $\gamma \geq 0$ satisfying \eqref{eq: gamma condition}. Then both $\{x_k\}$ and $\{z_k\}$ converge linearly to the unique solution $x_{*} \in K(x_{*})$ of the QVI \eqref{eq:QVI}.
\end{thm}

\begin{proof}
We obtain from Algorithm \ref{Alg:AlgL} that
\begin{eqnarray}\label{equi}
y_k&=& z_{k+1}+\Big(\frac{1-2\theta_k}{\theta_k}\Big)(z_{k+1}-z_k)\nonumber \\
   &=& (1-\theta_k)z_k+\theta_k x_k+\theta_k \Big(\frac{1-2\theta_k}{\theta_k}\Big)(x_k-z_k)\nonumber \\
   &=& (1-\theta_k)z_k+\theta_k x_k+(1-2\theta_k)(x_k-z_k) \nonumber \\
   &=& \Big((1-\theta_k)-(1-2\theta_k)\Big)z_k+(\theta_k+1-2\theta_k)x_k \nonumber\\
   &=& (1-\theta_k)x_k+\theta_k z_k.
\end{eqnarray}
Now, given the unique solution $x_{*}$ of \eqref{eq:QVI}, we obtain
\begin{eqnarray}\label{ee1}
  \|x_{k+1}-x_{*} \| &=& \|P_{K(y_k)}(w_k-\gamma \mathcal{A}(y_k))-P_{K(x_{*})}(x_{*}-\gamma \mathcal{A}(x_{*}))\| \nonumber \\
   &\leq&
  \|P_{K(y_k)}(w_k-\gamma \mathcal{A}(y_k))-P_{K(x_{*})}(y_k-\gamma \mathcal{A}(y_k))\|\nonumber \\
  &&+\|P_{K(x_{*})}(y_k-\gamma \mathcal{A}(y_k))-P_{K(x_{*})}(x_{*}-\gamma \mathcal{A}(x_{*}))\|\nonumber \\
   &\leq& \lambda\|y_k-x_{*}\|+\|y_k-x_{*}+\gamma (\mathcal{A}(x_{*})-\mathcal{A}(y_k))\|.
  \end{eqnarray}
By the fact that $\mathcal{A}$ is $\mu$-strongly monotone and $L-$Lipschitz continuous, we obtain
\begin{eqnarray}\label{ee2}
  \|y_k-x_{*}-\gamma (\mathcal{A}(x_{*})-\mathcal{A}(y_k))\|^2
  &=& \|y_k-x_{*}\|^2-2\gamma \langle \mathcal{A}(y_k)-\mathcal{A}(x_{*}), y_k-x_{*} \rangle \nonumber\\
  &&+\gamma^2\|\mathcal{A}(y_k)-\mathcal{A}(x_{*})\|^2 \nonumber\\
  &\leq& (1-2\mu\gamma+\gamma^2L^2)\|y_k-x_{*}\|^2.
  \end{eqnarray}
Combining \eqref{ee1} and \eqref{ee2}, we get
\begin{eqnarray}\label{ee3}
  \|x_{k+1}-x_{*}  \| &\leq& \lambda\|y_k-x_{*}\|+\sqrt{1-2\mu\gamma+\gamma^2L^2}\|y_k-x_{*}\|\nonumber\\
  &=& \beta\|y_k-x_{*}\|,
  \end{eqnarray}
  with
  \begin{equation}\label{beta}
    \beta:= \sqrt{1-2\mu\gamma+\gamma^2L^2}+\lambda \in (0,1).
  \end{equation}
By \eqref{equi} and Lemma \ref{lm2}, we obtain
\begin{eqnarray}\label{app4}
\|y_k-x_{*}\|^2&=& \|(1-\theta_k)(x_k-x_{*})+\theta_k(z_k-x_{*})\|^2\nonumber\\
&=&(1-\theta_k)\|x_k-x_{*}\|^2+\theta_k\|z_k-x_{*}\|^2\nonumber\\
&&-\theta_k(1-\theta_k)\|x_k-z_k\|^2
\end{eqnarray}
and
\begin{eqnarray}\label{app5}
\|z_{k+1}-x_{*}\|^2&=& \|(1-\theta_k)(z_k-x_{*})+\theta_k(x_k-x_{*})\|^2\nonumber\\
&=&(1-\theta_k)\|z_k-x_{*}\|^2+\theta_k\|x_k-x_{*}\|^2\nonumber\\
&& -\theta_k(1-\theta_k)\|x_k-z_k\|^2.
\end{eqnarray}
By these last two identities, we obtain from \eqref{ee3} that
\begin{eqnarray}\label{app4pre}
\|x_{k+1}-x_{*}\|^2+\|z_{k+1}-x_{*}\|^2&\leq& \beta^2\|y_k-x_{*}\|^2 \nonumber\\
&&+(1-\theta_k)\|z_k-x_{*}\|^2+\theta_k\|x_k-x_{*}\|^2 \nonumber\\
&&-\theta_k(1-\theta_k)\|x_k-z_k\|^2\nonumber\\
&=&(1-\theta_k)\beta^2\|x_k-x_{*}\|^2+\theta_k\beta^2\|z_k-x_{*}\|^2\nonumber\\
&&-\beta^2\theta_k(1-\theta_k)\|x_k-z_k\|^2\nonumber\\
&&+(1-\theta_k)\|z_k-x_{*}\|^2+\theta_k\|x_k-x_{*}\|^2\nonumber\\
&&-\theta_k(1-\theta_k)\|x_k-z_k\|^2\nonumber\\
&=&\beta^2\Big(\|x_k-x_{*}\|^2+\|z_k-x_{*}\|^2 \Big)\nonumber\\
&&+\theta_k(1-\beta^2)\|x_k-x_{*}\|^2+(1-\theta_k(1-\beta^2)\|z_k-x_{*}\|^2\nonumber\\
&&-\theta_k(1-\theta_k)(1+\beta^2)\|x_k-z_k\|^2\nonumber\\
&\leq& \beta^2\Big(\|x_k-x_{*}\|^2+\|z_k-x_{*}\|^2 \Big)\nonumber\\
&&+\theta_k(1-\beta^2)\|x_k-x_{*}\|^2+(1-\beta^2)(1-\theta_k)\|z_k-x_{*}\|^2\nonumber\\
&\leq& \beta^2\Big(\|x_k-x_{*}\|^2+\|z_k-x_{*}\|^2 \Big)\nonumber\\
&&+\theta_k(1-\beta^2)\Big(\|x_k-x_{*}\|^2+\|z_k-x_{*}\|^2 \Big)\nonumber\\
&&+(1-\beta^2)(1-\theta_k)\Big(\|x_k-x_{*}\|^2+\|z_k-x_{*}\|^2 \Big)\nonumber\\
&\leq& \beta^2\Big(\|x_k-x_{*}\|^2+\|z_k-x_{*}\|^2 \Big)\nonumber\\
&&+b(1-\beta^2)\Big(\|x_k-x_{*}\|^2+\|z_k-x_{*}\|^2 \Big)\nonumber\\
&&+(1-\beta^2)(1-a)\Big(\|x_k-x_{*}\|^2+\|z_k-x_{*}\|^2 \Big)\nonumber\\
&=&\varrho \Big(\|x_k-x_{*}\|^2+\|z_k-x_{*}\|^2 \Big),
\end{eqnarray}
where
$$
\varrho:= \max\Big\{\beta^2,b(1-\beta^2),(1-\beta^2)(1-a) \Big\} \in (0,1).
$$
\noindent Consequently,
\begin{eqnarray}\label{oppos}
\|x_{k+1}-x_{*}\|^2&\leq&\|x_{k+1}-x_{*}\|^2+\|z_{k+1}-x_{*}\|^2 \nonumber\\
&\leq&\varrho \Big(\|x_k-x_{*}\|^2+\|z_k-x_{*}\|^2 \Big)\nonumber \\
&\leq&\varrho^{k+1} \Big(\|x_0-x_{*}\|^2+\|z_0-x_{*}\|^2 \Big)\nonumber \\
&=&2\varrho^{k+1} \|x_0-x_{*}\|^2.
\end{eqnarray}
Hence, $\{x_k\}$  converges linearly to the unique solution $x_{*} \in K(x_{*})$ of the QVI \eqref{eq:QVI}. We can also show from \eqref{oppos} that $\{z_k\}$ also converges linearly to the unique solution $x_{*} \in K(x_{*})$ of the QVI \eqref{eq:QVI} since
\begin{eqnarray}\label{oppos2a}
\|z_{k+1}-x_{*}\|^2&\leq&\|x_{k+1}-x_{*}\|^2+\|z_{k+1}-x_{*}\|^2 \nonumber\\
&\leq&2\varrho^{k+1} \|x_0-x_{*}\|^2.
\end{eqnarray}
\end{proof}

\noindent In a special case when $K(x), x \in H$ is a "moving set". That is, the case when $K(x):=c(x)+K, x \in H$ where $c:H \rightarrow H$ is a $\lambda$-Lipschitz continuous mapping and $K\subset H$ is a nonempty, closed and convex subset. Then the Assumption \eqref{eq: gamma condition} is automatically satisfied with the same value of $\lambda$ (see \cite{Nesterov}). The following result hold in this case.

\begin{cor}
Consider the QVI \eqref{eq:QVI} with $\mathcal{A}$ being $\mu$-strongly monotone and $L$-Lipschitz continuous and suppose that $K(x):=c(x)+K, x \in H$ where $c:H \rightarrow H$ is a $\lambda$-Lipschitz continuous mapping and $K$ is a nonempty, closed and convex subset of $H$.
Let $\{x_k\}$ be any sequence generated by Algorithm \ref{Alg:AlgL} with $\gamma \geq 0$ satisfying \eqref{eq: gamma condition}. Then $\{x_k\}$ and $\{z_k\}$ converge linearly to the unique solution $x_{*} \in K(x_{*})$ of the QVI \eqref{eq:QVI}.
\end{cor}

\begin{rem}$\,$\\
(a) In the inertial type gradient projection algorithms proposed in \cite{Copur,Jabeen,ShehuOptim}, only strong convergence results are obtained with no linear rate of convergence discussed. In this paper, we propose an inertial type gradient projection algorithm with linear convergence results given. \\[-1mm]

\noindent
(b) In the case when $\theta_k=\frac{1}{2}$ in our Algorithm \ref{Alg:AlgL}, our algorithm becomes a special case of the algorithms studied in
 \cite{Mijajlovic2018,ShehuOptim}.
\end{rem}

\section{Numerical Examples}\label{Sec:Numerics}
\noindent We give some numerical implementations of our proposed Algorithm \ref{Alg:AlgL} and give comparisons with some existing methods in the literature.   All codes were written in MATLAB R2023a and performed on a PC Desktop Intel Core i5-8265U CPU \@ 1.60GHz   1.80 GHz, RAM 16.00GB. We compare Algorithm \ref{Alg:AlgL} with the proposed algorithms in \cite{ajm18,Copur,Shehu,ShehuOptim}. \\

\noindent
We choose to use the test problem library QVILIB taken from \cite{fks12}; the feasible map $K$ is assumed to be given by
$K(x):=\{z \in \mathbb{R}^n \, : \, g(z,x) \leq 0 \}$.
We implemented Algorithm \ref{Alg:AlgL} in Matlab. We implemented the projection over a convex set as the solution of a convex program. We considered the following performance measures for optimality and feasibility
$$
\text{opt}(x) := - \min_z \{{\cal F}(x)^T (z-x) \, : \, z \in K(x) \}, \;\; \text{feas}(x) := \|\max\{0,g(x,x)\}\|_\infty.
$$
A point $x^*$ is considered a solution of the QVI if the optimality measure opt$(x^*) \leq 1e$-4 and feasibility measure feas$(x^*) \leq 1e$-4. For solving the QVI, we utilized the built-in function \texttt{fmincon} with `sqp' algorithm as its internal method, setting the maximum iteration count to 1000. The QVILIB \cite{fks12} provides a comprehensive collection of test problems specifically designed for evaluating algorithms employed in solving QVI. These problems encompasses a wide range of scenarios, including academic models, real-world applications and discretized versions of infinite-dimensional QVIs that model various engineering and physical phenomena. Furthermore, the library offers an M-file named \texttt{startinPoints}, which facilitates obtaining initial strating points for each test problem.
  The feasible set $K(x)$ is defined as the intersection of a fixed set $\bar{K}$ and a moving set $\tilde{K}(x)$ that depends on the point $x$ given by:
\begin{eqnarray*}
	\bar{K} &:=& \{y \in \mathbb{R}^k \; | \; g^{I}(y) \leq 0,\; M^{I}y + v^{I} =0\}, \\
	\tilde{K}(x)&=& \{y \in \mathbb{R}^k \; | \; g^{P}(y,x) \leq 0, \; M^{P}(x)y + v^{P}(x) = 0 \}.
\end{eqnarray*}
\noindent The comprehensive definitions of each problem can be found in \cite{fks12}; however, we provide a brief description of each problem in Table \ref{tab_problem}.

\begin{table}[h!]
	\centering
	\caption{Description of test problems used in the experiments}\label{tab_problem}
	\begin{tabular}{ccccccc}
		\toprule
		Problem name & $n$ & $m_I$ & $p_I$ & $m_P$ & $p_P$ & n(start) \\
		\midrule
		\texttt{OutZ40} & 2 & 4 & 0 & 2 & 0 & 3 \\
		\texttt{OutZ41} & 2 & 4 & 0 & 2 & 0 & 3 \\
		\texttt{OutZ42} & 4 & 4 & 0 & 4 & 0 & 4 \\
		\texttt{OutZ43} & 4 & 0 & 0 & 4 & 0 & 3 \\
		\texttt{OutZ44} & 4 & 0 & 0 & 4 & 0 & 3 \\
		\texttt{MovSet1A} & 5 & 0 & 0 & 1 & 0 & 2 \\
		\texttt{MovSet1B} & 5 & 0 & 0 & 1 & 0 & 2 \\
		\texttt{MovSet2A} & 5 & 0 & 0 & 1 & 0 & 2 \\
		\texttt{MovSet2B} & 5 & 0 & 0 & 1 & 0 & 2 \\
		\texttt{Box1A} & 5 & 0 & 0 & 10 & 0 & 2 \\
		\texttt{Box1B} & 5 & 0 & 0 & 10 & 0 & 2 \\
		\texttt{BiLin1A} & 5 & 10 & 0 & 3 & 0 & 2 \\
		\texttt{BiLin1B} & 5 & 10 & 0 & 3 & 0 & 2 \\
		\texttt{WalEq1} & 18 & 18 & 1 & 5 & 0 & 2 \\
		\texttt{WalEq2} & 105 & 105 & 1 & 20 & 0 & 2 \\
		\texttt{WalEq3} & 186 & 186 & 1 & 30 & 0 & 2 \\
		\texttt{WalEq4} & 310 & 310 & 1 & 30 & 0 & 2 \\
		\texttt{WalEq5} & 492 & 492 & 1 & 40 & 0 & 2 \\
		\texttt{Wal2} & 105 & 107 & 0 & 20 & 0 & 2 \\
		\texttt{Wal3} & 186 & 188 & 0 & 30 & 0 & 2 \\
		\texttt{LunSSVI1} & 501 & 1002 & 1 & 0 & 0 & 2 \\
		\texttt{OutKZ41} & 82 & 82 & 0 & 82 & 0 & 2 \\
		\bottomrule
	\end{tabular}
\end{table}
In Table \ref{tab_problem}, the first column contains the name of the problem, the second column ($n$) contains the number of variables in the problem, column $m_{I}$ contains the number of inequality constraints defining $\bar{K}$, column $p_{I}$ contains the number of linear equalities in $\bar{K},$ the column $m_P$ contains the number of inequality constraints defining $\tilde{K}(x)$, the column the $p_{P}$ is the number of equalities in the definition of $\tilde{K}(x)$ and the last column n(start) is the number of starting points for the problem.

In order to compare the performance of the algorithms, we used the so-called performance profiles, see \cite{DolanMore2002}, based on the number of iterations and execution time of the algorithm. Let $t_{s,i}>0$ denote the comparison metric (related to the current performance index) of solver $s \in \mathcal S$ for addressing instance $i \in \mathcal I$ of the problem, where $\mathcal S$ represents the set of comparing algorithms and $\mathcal I$ denotes the various problems with different starting points.
We define the performance ratio by 
\begin{equation*}
	\forall s\in\mathcal S,\,\forall i\in\mathcal I\colon\quad
	r_{s,i} := \dfrac{t_{s,i}}{\min\{t_{s',i}\,|\, s' \in \mathcal S\}}.
\end{equation*}
This quantity is called the ratio of the performance of solver $s\in \mathcal S$ to solve problem $i\in \mathcal I$ 
compared to the best performance of any other algorithm in $\mathcal S$ to solve problem $i$. 
The cumulative distribution function $\rho_s\colon[1,\infty)\to[0,1]$ of the current performance index linked with solver $s\in\mathcal S$ is defined as follows:
\begin{equation*}
	\forall \tau \in[1,\infty)\colon\quad
	\rho_{s}(T) := \frac{|\{i \in \mathcal I\,|\, r_{s,i} \leq T\}|}{|\mathcal I|}.
\end{equation*}
The performance profile for a specific performance index displays the graphical representations of all the functions $\rho_{s}$, where $s$ varies across the set $\mathcal{S}$.
The value $\rho_{s}(1)$ indicates the proportion of problem instances where solver $s\in\mathcal S$ exhibits the best performance. For any arbitrary $\tau\geq 1$, $\rho_s(T)$ represents the fraction of problem instances where solver $s\in\mathcal S$ achieves at most $T$ times the best performance.

In our experiments, we choose the following control parameters for the algorithms:
\begin{itemize}
	\item Proposed Alg: $\gamma = 0.5, \theta_k = \frac{k}{5(k+1)};$
	\item Antipin et al. Alg: $\gamma = 0.5;$
	\item Mijajlovic et al Alg. 1: $\gamma = 0.5, \alpha_k = \frac{1}{k+1};$
	\item Mijajlovic et al. Alg. 2: $\gamma = 0.5, \alpha_k = \frac{1}{k+1}, \beta_k = \frac{3k}{7k+9};$
	\item Shehu et al Alg: $\gamma = 0.5, \theta_k = \frac{k}{5(k+1)}, \alpha_k = \frac{1}{k+1}.$
\end{itemize}

Figure \ref{fig1} provides a comprehensive view of the performance profile analysis, focusing on the number of iterations required by each algorithm.  The result highlights the proposed Algorithm as the standout performer, demonstrating superior efficiency with the least number of iterations across nearly 82\% of the experimental scenarios. This suggests that the proposed Algorithm consistently converges more rapidly compared to its counterparts, making it a compelling choice for solving the QVI problems. In comparison, the Shehu et al Algorithm emerges as the second-best performer, showcasing the best performance in approximately 64\% of the experiments. While not as dominant as the proposed Algorithm, it shares similar number of iterations with the proposed algorithm in significant number of cases. Furthermore, the Antipin et al Algorithm achieves the best performance in roughly 57\% of the experiments, the Mijajlovic et al Algorithm 2 has best performance in 53\% of the experiments, while Mijajlovic et al Algorithm 1 has best performance in about 51\% of the experimental scenarios. This is collaborated with the average results shown in Table \ref{tab2}. Although Antinpin et al. Alg has lower average number of iteration than Mijajlovic et al Alg 1 and 2, this is because Antipin et al. Alg. has smaller figures at instances it achieved successes compare to the Mijajlovic et al algorithms.

Figure \ref{fig2} shows the performance profile of the algorithm based on the time of execution. The result shows that the proposed algorithm has the least time of execution for about 45\% of the scenarios. This is follow by Shehu et al algorithm which shares very closed figures with the rest of the comparing algorithms. This further highlights the advantage of the proposed algorithm over the rest of the algorithms. The average execution time presented in Table \ref{tab2} indicates that the Antipin et al Algorithm exhibits the longest average execution time. This outcome aligns with expectations, as the Antipin et al. Algorithm involves computing two projections per iteration utilizing the optimization solver, which inherently consumes additional computational resources. Similarly, the Mijajlovic et al Algorithm 2 also computes two projections in each iteration, contributing to a higher average execution time. However, it's worth noting that in this case, the output is regularized by the control parameters $\alpha_{k}$ and $\beta_k$, which may mitigate some of the computational overhead associated with the additional projections.

Figure \ref{fig3} shows the number of instances the algorithm terminate due to the stopping criterion. Recall that the algorithm is terminated if a solution of the QVI problem is found (Termination 1) or the computation reaches the maximum iteration (Termination 2). From Figure \ref{fig3}, we observe that the Proposed Algorithm terminated due to Termination 1 in 39 instances, which accounts for approximately 88.6\% of the total cases. Similarly, Mijajlovic et al. Algorithms 1 and 2, along with Antipin et al. Algorithm, stopped due to Termination 1 in 37 instances, representing 84.1\% of the experiments. Additionally, the Shehu et al. Algorithm reached Termination 1 in 36 instances, constituting about 81.8\% of the total scenarios. These findings underscore the computational advantage of the Proposed Algorithm, as it consistently terminates due to finding a solution to the QVI problem in a higher proportion of instances compared to the other algorithms. This also suggests that the Proposed Algorithm converges more efficiently, leading to quicker termination based on achieving the desired solution.

Figure \ref{fig4} shows the result of the experimental order of convergence (EOC) of each algorithm. The EOC is calculated by computing
\[  EOC := \max \left(\frac{\log \|x_{k+1} - x_{k}\|}{\log \|x_{k}-x_{k-1}\|}, \frac{\log \|x_{k+2}-x_{k+1}\|}{\log \|x_{k+1}-x_{k}\|}\right). \]
This formula provides an estimation of how quickly the error (or performance measure) decreases as the algorithm progresses. A higher value of EOC indicates faster convergence. From Figure \ref{fig4}, we observe that the proposed algorithm achieves the highest EOC value in 11 instances. In 31 instances, it shares equivalent EOC values with other algorithms, indicating comparable convergence rates. However, in 2 instances, the proposed algorithm exhibits a lower EOC value. Interestingly, Mijajlovic et al Algorithm 2 attains the highest EOC value in only 1 instance. This outlier occurrence could potentially be attributed to the influence of regularization parameters integrated within the algorithm, which may have impacted its convergence behavior in this particular case.

\begin{table}[h!]
	\centering
	\caption{Summary of numerical results}\label{tab2}
	\begin{tabular}{ccccc}
		\toprule
		 & Av. Iter & Av. Time  & \# Feas sol. & \# No Feas. sol. \\
		\midrule
		Proposed Alg. & 4.5455 & 4.6086 & 39 & 5 \\
		Mijajlovic et al Alg. 1 & 6.7045 & 8.5337 & 37 & 7 \\
		Mijajlovic et al Alg. 2 & 7.3863 & 14.8086 & 37 & 7 \\
		Shehu et al Alg. & 6.1818 & 6.2171 & 36 & 8 \\
		Antipin et al Alg. & 7.8409 & 18.2486 & 37 & 7 \\
		\bottomrule
	\end{tabular}
\end{table}

 \begin{figure}
 	\begin{center}
 		\includegraphics[height=8.0cm]{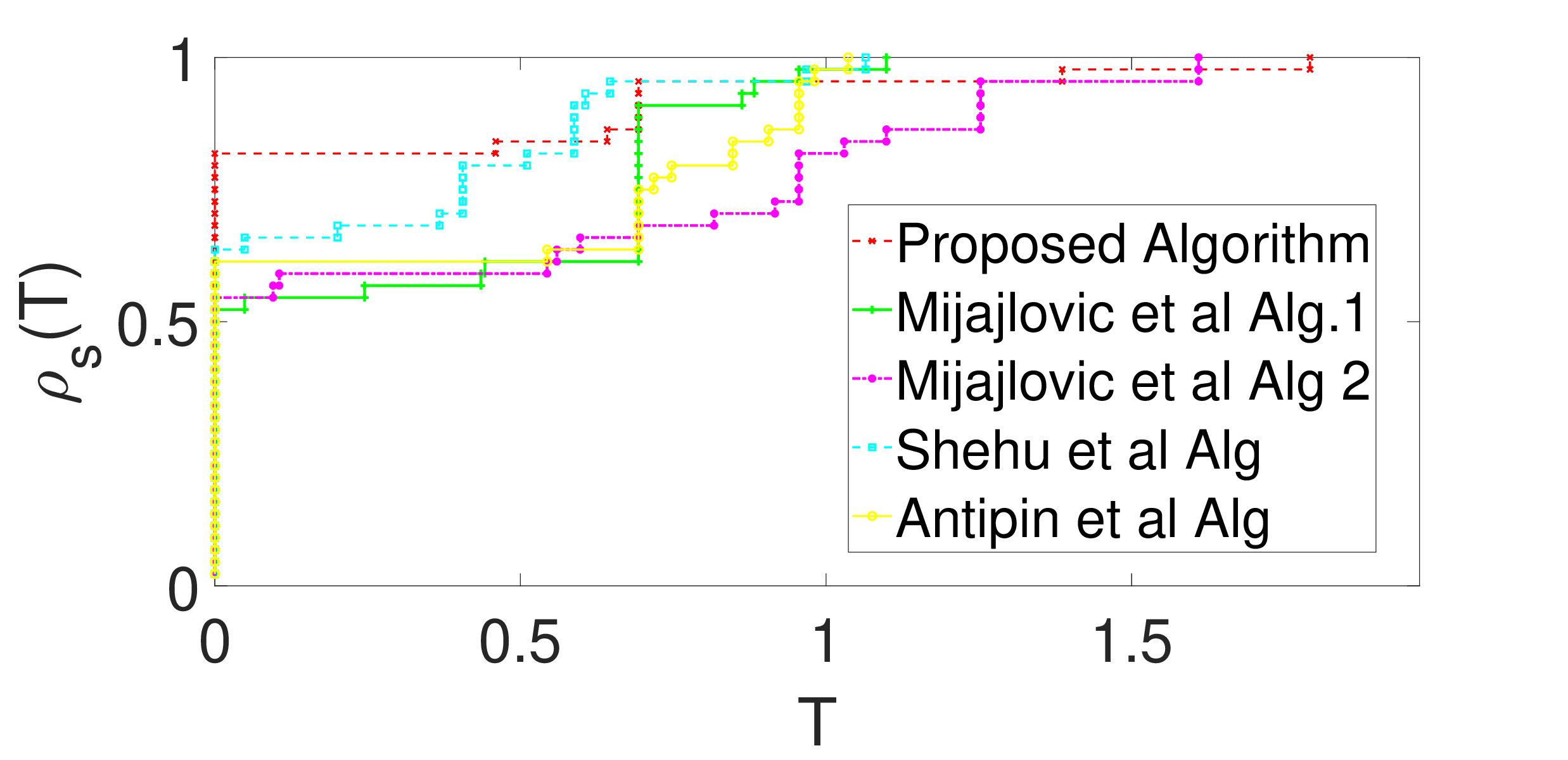}	
 		\caption{Performance profile based on number of iterations}	\label{fig1}
 	\end{center}
 \end{figure}

 \begin{figure}
	\begin{center}
		\includegraphics[height=8.0cm]{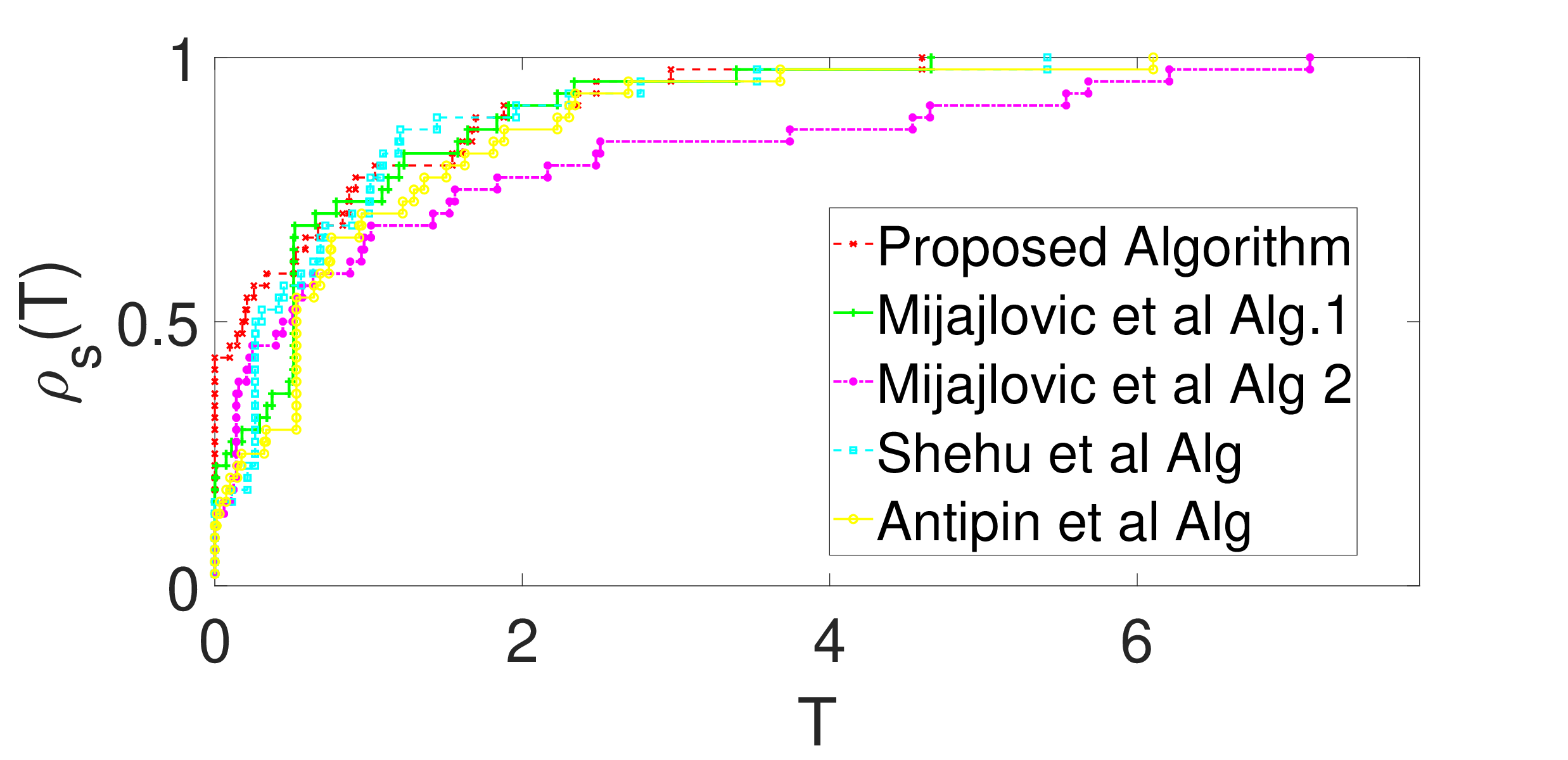}	
		\caption{Performance profile based on the time of execution}	\label{fig2}
	\end{center}
\end{figure}

 \begin{figure}
	\begin{center}
		\includegraphics[height=8.0cm]{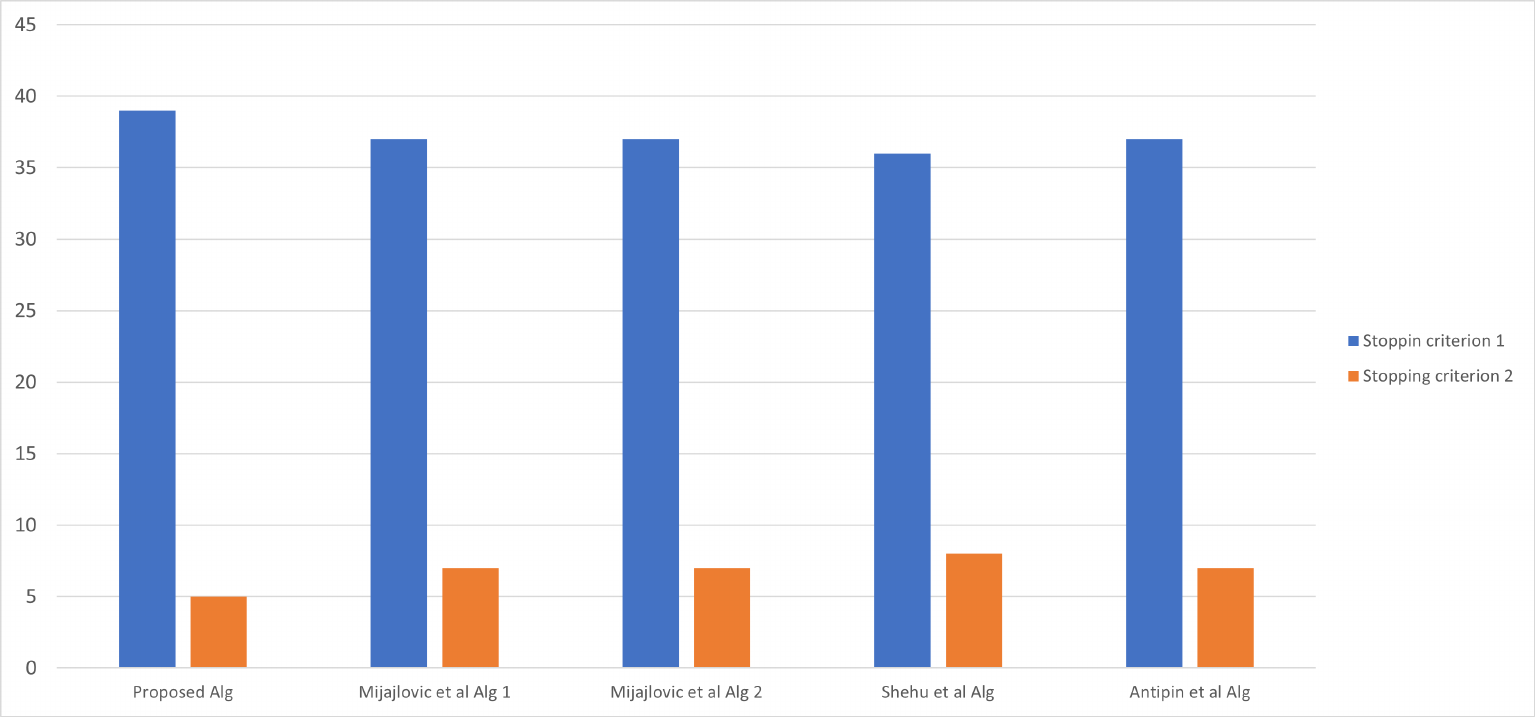}	
		\caption{Numerical results showing the feasibility of solution based on the reason of termination}	\label{fig3}
	\end{center}
\end{figure}

\begin{figure}
	\begin{center}
		\includegraphics[height=8.0cm]{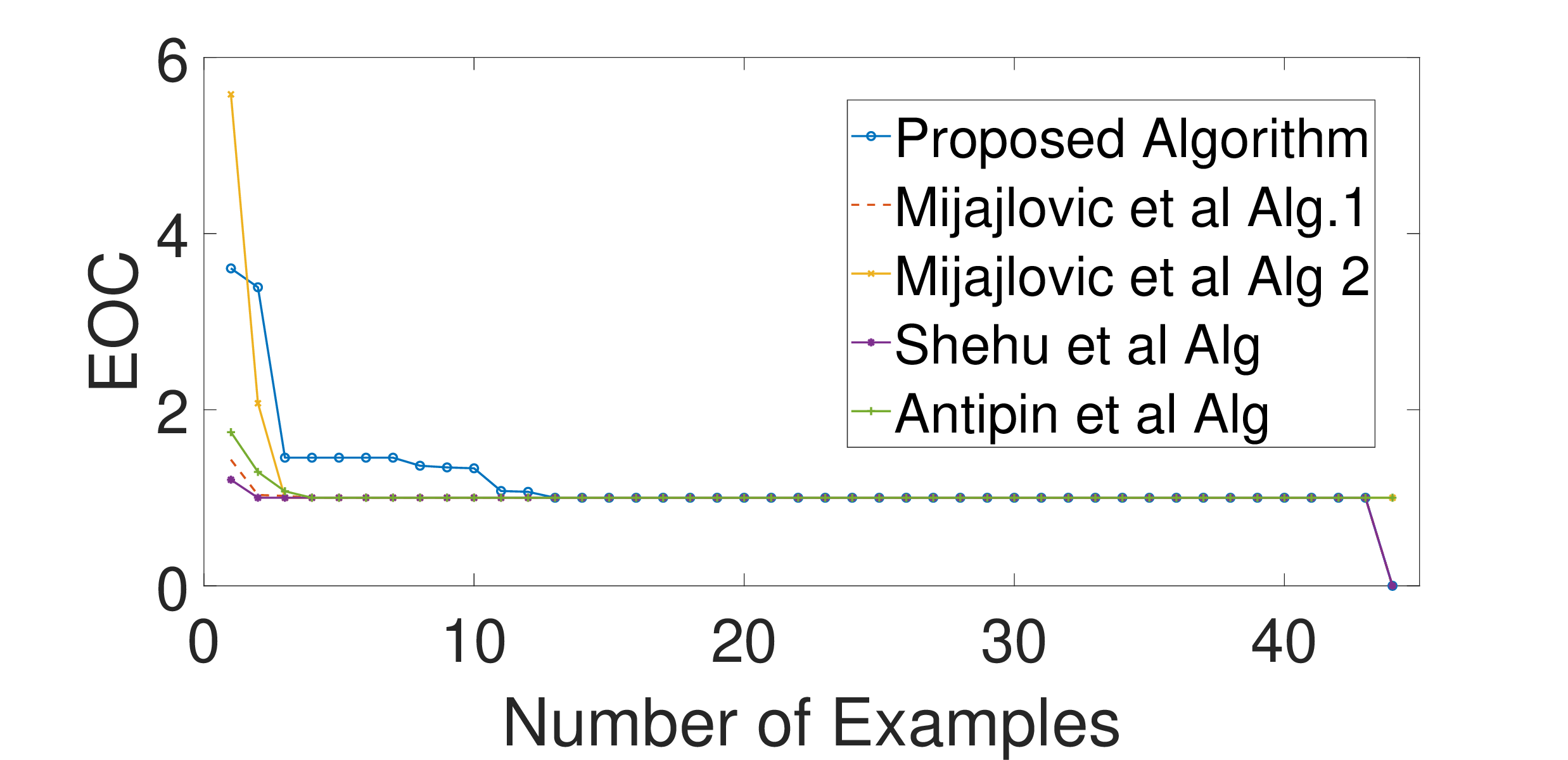}	
		\caption{Experimental order of convergence of algorithms}	\label{fig4}
	\end{center}
\end{figure}

\section{Conclusion}\label{Sec:Final}
\noindent
In this paper, we introduced an inertial type gradient projection algorithm for solving quasi-variational inequalities in Hilbert spaces and obtain its linear convergence rate under strong monotonicity of the operator.  This result complements other results in the literature for inertial type gradient projection algorithms to solve quasi-variational inequalities where only strong convergence results are obtained.
We also showed that the inertial factor in our proposed inertial type gradient projection algorithm could take both negative and non-negative values unlike many other inertial type gradient projection algorithms for quasi-variational inequalities in the literature.
The numerical comparisons of the proposed algorithm showed that our proposed inertial type gradient projection algorithm is efficient and outperform some popular related inertial type gradient projection algorithms in the literature for quasi-variational inequalities.

\section*{Disclosure statement}

\subsection*{Ethical Approval and Consent to participate}
All the authors gave ethical approval and consent to participate in this article.
\subsection*{Consent for publication}
All the authors gave consent for the publication of identifiable details to be published in the journal and article.

\subsection*{Code availability} The Matlab codes employed to run the numerical experiments are available on request.

\subsection*{Availability of supporting data}
Data sharing is not applicable to this article as no datasets were generated or analyzed
during the current study.
\subsection*{Competing interests}
The authors declare no competing interests.
\subsection*{Funding}
Not Applicable.

\subsection*{Authors' contributions}
Y.Y. and Y.S. wrote the manuscript and L.O.J. prepared the all the figures and tables.

\subsection*{Acknowledgments}
Not Applicable.

\end{document}